\documentclass[11pt]{article}
\usepackage[a4paper,margin=25mm]{geometry}
\usepackage[T1]{fontenc}
\usepackage{lmodern}
\usepackage{amsmath,amssymb,amsthm,mathtools}
\usepackage{booktabs,microtype,xcolor,textcomp,longtable}
\usepackage{setspace}
\usepackage[colorlinks=true,linkcolor=blue!55!black,citecolor=blue!55!black,urlcolor=blue!55!black]{hyperref}
\newtheorem{theorem}{Theorem}[section]
\newtheorem{proposition}[theorem]{Proposition}
\newtheorem{lemma}[theorem]{Lemma}
\newtheorem{corollary}[theorem]{Corollary}
\theoremstyle{remark}

\numberwithin{equation}{section}
\newcommand{\fall}[2]{#1^{\underline{#2}}}
\allowdisplaybreaks
\hypersetup{
  pdftitle={Infinite log-concavity of the Boros--Moll sequences},
  pdfauthor={Matthew H. Y. Xie and Philip B. Zhang},
  pdfkeywords={Boros-Moll polynomials, infinite log-concavity, real-rootedness, Narayana polynomials, Jacobi polynomials}
}
\title{Infinite log-concavity of the Boros--Moll sequences}
\author{Matthew H. Y. Xie\textsuperscript{1} and
Philip B. Zhang\textsuperscript{2}\\[6pt]
  \small \textsuperscript{1}School of Mathematical Sciences,\\
  \small Tianjin University of Technology, Tianjin 300384, P. R. China\\[4pt]
  \small \textsuperscript{2}College of Mathematical Sciences
   \& Institute of Mathematics and Interdisciplinary Sciences,\\
  \small Tianjin Normal University, Tianjin 300387, P. R. China\\[6pt]
  \small Email: \textsuperscript{1}\href{mailto:xie@email.tjut.edu.cn}{\texttt{xie@email.tjut.edu.cn}},
  \textsuperscript{2}\href{mailto:zhang@tjnu.edu.cn}{\texttt{zhang@tjnu.edu.cn}} (corresponding author)}
\date{}
\begin{document}
\setstretch{1.2}
\maketitle
\begin{abstract}
Let $(d_i(n))_{i=0}^n$ be the Boros--Moll coefficient sequence.
We prove that, for every integer $n\ge1$,
the polynomial
\[
 M_n(x)=\sum_{i=0}^n
 \bigl(d_i(n)^2-d_{i-1}(n)d_{i+1}(n)\bigr)x^i
\]
has only simple negative zeros, which strictly interlace those of the
Narayana polynomial of the same degree. This proves a conjecture of
Chen, Yang, and Zhang and,
by Br\"and\'en's preservation theorem, settles the infinite
log-concavity conjecture of Boros and Moll. The proof uses an expansion
of the reversed and normalized form of $M_n(x)$ in derivatives of
the Narayana polynomial, together with estimates for the weights and
partial sums of the normalized derivatives.
\end{abstract}

\noindent\textbf{2020 Mathematics Subject Classification.}
Primary 26C10; Secondary 05A20, 33C45.

\smallskip
\noindent\textbf{Keywords.}
Boros--Moll polynomials, infinite log-concavity, real-rootedness,
Narayana polynomials, Jacobi polynomials.

\section{Introduction}

The Boros--Moll polynomials arise in the evaluation of the quartic integral
\[
 \int_0^\infty\frac{dt}{(t^4+2at^2+1)^{n+1}}
 =\frac{\pi}{2^{n+3/2}(a+1)^{n+1/2}}B_n(a),
 \qquad a>-1,
\]
where $n$ is a nonnegative integer and
$B_n(x)=\sum_{i=0}^n d_i(n)x^i$
\cite{AmdeberhanMoll2009,BorosMollIntegral,BorosMollDoubleRoot2001,BorosMollBook}.
The coefficients are positive and admit the expression
\begin{equation}\label{eq:bm-coefficients}
 d_i(n)=4^{-n}\sum_{k=i}^n2^k
 \binom{2n-2k}{n-k}\binom{n+k}{k}\binom{k}{i},
 \qquad 0\le i\le n.
\end{equation}
We set $d_i(n)=0$ for $i<0$ or $i>n$.
We study the coefficient sequence under repeated applications of the
log-concavity operator. For a finite nonnegative sequence
$\mathbf{a}=(a_0,\ldots,a_n)$, extend $\mathbf{a}$ by zero and define
\[
 (\mathcal L\mathbf{a})_i=a_i^2-a_{i-1}a_{i+1}.
\]
We say that $\mathbf{a}$ is $r$-log-concave if $\mathcal L^j\mathbf{a}$
is nonnegative for $1\le j\le r$, and infinitely log-concave if this
holds for every $r\ge1$.

Boros and Moll \cite{BorosMoll1999} proved the unimodality
of the coefficient sequence $(d_i(n))_{i=0}^n$.
Their ordinary log-concavity conjecture, recorded by Moll
\cite[p.~316]{Moll2002}, was proved by Kauers and Paule
\cite{KauersPaule2007} using recurrences for $d_i(n)$.
Chen, Pang, and Qu gave combinatorial proofs
of the coefficient formula \cite{ChenPangQuFormula} and of
log-concavity \cite{ChenPangQuPermutations}. Boros and Moll further
conjectured that $(d_i(n))_{i=0}^n$ is infinitely log-concave for every
$n\ge0$, as recorded by Moll \cite[p.~316]{Moll2002}; see also
\cite[p.~157]{BorosMollBook} and
\cite[Conjecture~8.4]{Branden}.

Chen and Xia \cite{ChenXia2LC} proved that the Boros--Moll coefficients
are $2$-log-concave by bounding the ratios of consecutive coefficients.
Chen, Dou, and Yang \cite{ChenDouYang2013}
proved Br\"and\'en's conjectures that
$\sum_{i=0}^n d_i(n)x^i/i!$ and
$\sum_{i=0}^n d_i(n)x^i/(i+2)!$ have only real zeros.
In fact, they showed that both polynomial families form Sturm sequences
\cite[Theorems~2.2 and~2.4]{ChenDouYang2013}.
By the results of Craven and Csordas on iterated Tur\'an inequalities
\cite{CravenCsordas}, these real-rootedness results imply
$2$- and $3$-log-concavity of the Boros--Moll coefficient sequence, respectively.
Guo \cite{GuoTuran2022} established higher order Tur\'an
inequalities for the Boros--Moll sequences.
Further coefficient inequalities appear in
\cite{ChenGuReverse,ChenXiaRatio,ChenWangXiaInterlacing,ChenXiaMinimum}.
For fixed $i$, Jiang and Wang \cite{JiangWang2024} studied
higher order Tur\'an inequalities and Laguerre inequalities of low order
for the sequences $(d_i(n))_{n\ge i}$.

One sufficient condition for infinite log-concavity is
$a_i^2\ge r_0a_{i-1}a_{i+1}$, where $r_0=(3+\sqrt5)/2$.
For nonnegative sequences, this condition is preserved by $\mathcal L$
\cite{CravenCsordas,McNamaraSagan}.
Uminsky and Yeats \cite{UminskyYeats2007} constructed invariant regions
for symmetric sequences with $a_0=a_n=1$.
According to McNamara and Sagan~\cite[Section~2]{McNamaraSagan},
Kauers verified the infinite log-concavity conjecture of Boros and Moll
computationally for $n\le129$ by checking that the fifth iterate satisfies
this strengthened condition.

A second route to infinite log-concavity uses real-rootedness.
Br\"and\'en \cite{Branden} proved that if a finite
nonnegative sequence has a generating polynomial with only
nonpositive real zeros, then so does its image under $\mathcal L$.
This settled a conjecture proposed independently by Fisk, McNamara--Sagan,
and Stanley; see~\cite{Fisk} and~\cite[Conjecture~7.1]{McNamaraSagan},
where Stanley's contribution is recorded as a private communication.
For the Boros--Moll coefficients, this theorem cannot be applied directly:
$B_n(x)$ has no real zeros when $n$ is even and exactly one when $n$
is odd \cite{BorosMoll1999}, and hence is not real-rooted for $n\ge2$.
Dimitrov
\cite{Dimitrov2004} and Lopez Santander, McLaughlin, and Moll
\cite{LopezSantanderMcLaughlinMoll2026} studied the zeros and asymptotics
of $B_n(x)$.

Chen, Yang, and Zhang
\cite[Conjecture~1.3]{ChenYangZhang2018} proposed applying the preservation
theorem after the first iteration. More precisely, they conjectured that
the polynomial
\begin{equation}\label{eq:first-transform}
 M_n(x)=\sum_{i=0}^n
 \bigl(d_i(n)^2-d_{i-1}(n)d_{i+1}(n)\bigr)x^i
\end{equation}
has only real zeros for every $n\ge1$.
They introduced generalized Narayana polynomials, proved their
real-rootedness by the criterion of Liu and Wang \cite{LiuWang2007},
and expressed $M_n(x)$ as a positive linear combination of these polynomials.

We prove the real-rootedness conjecture of Chen, Yang, and Zhang by
comparing $M_n(x)$ with the classical Narayana polynomial of the same degree,
\[
 N_n(x)=\sum_{k=0}^n N_{n,k}x^k,\qquad
 N_{n,k}=\frac1{k+1}\binom nk\binom{n+1}k\quad(0\le k\le n).
\]
Since $N_{n,k}=\binom nk^2-\binom n{k-1}\binom n{k+1}$, $N_n(x)$
generates the first transform of the binomial coefficient sequence.
The polynomials $N_n(x)$ also occur in Br\"and\'en's proof
of the preservation theorem \cite[Section~3]{Branden}.
The Jacobi representation of $N_n(x)$, discussed by Kostov,
Mart\'inez-Finkelshtein, and Shapiro
\cite[Proposition~6]{KostovMartinezFinkelshteinShapiro2009},
shows that $N_n(x)$ has $n$ simple negative zeros.

\noindent\begin{minipage}{\linewidth}
We prove the following stronger form of the conjecture of Chen, Yang, and Zhang.

\begin{theorem}\label{thm:main}\label{thm:interlacing}
For every integer $n\ge1$, the polynomial $M_n(x)$ has $n$ distinct
negative real zeros. Writing $\mu_{n,1}<\cdots<\mu_{n,n}$ and
$\nu_{n,1}<\cdots<\nu_{n,n}$ for the zeros of $M_n(x)$ and $N_n(x)$,
respectively, we have
\[
 \mu_{n,1}<\nu_{n,1}<\cdots<\mu_{n,n}<\nu_{n,n}<0.
\]
\end{theorem}
\end{minipage}

The negative real zeros of $M_n(x)$ in Theorem~\ref{thm:main} allow us to
apply Br\"and\'en's preservation theorem after the first transform,
settling the infinite log-concavity conjecture of Boros and Moll.

\begin{corollary}\label{thm:infinite-log-concavity}
For integers $n,r\ge1$, the generating polynomial of
$\mathcal L^r(d_0(n),\ldots,d_n(n))$ has only negative real zeros
and all its coefficients are strictly positive. Consequently, for every
integer $n\ge0$, the sequence $(d_i(n))_{i=0}^n$ is infinitely log-concave.
\end{corollary}

After reversing and normalizing $M_n(x)$, we expand the resulting
polynomial in derivatives of $N_n(x)$. At a zero of $N_n(x)$, it suffices
to prove that a weighted sum of normalized derivatives is positive.
For $n\ge36$, the weights increase in the first two steps and then decrease.
We estimate these two increases and prove a uniform lower bound for the
partial sums of the normalized derivatives. Abel summation gives the
required sign. The remaining degrees $1\le n\le35$ are verified by exact
polynomial division.

This paper is organized as follows. Section~\ref{sec:zero-comparison}
reduces the sign comparison to two estimates. We prove the coefficient
estimate in Section~\ref{sec:coefficient-comparison} and the bound for the
partial sums in Section~\ref{sec:partial-sum-bounds}. The proofs of the
main results are completed in Section~\ref{sec:main-proofs}.

\section{Reduction to two estimates}\label{sec:zero-comparison}

For $n\ge1$, we reverse and normalize the coefficients by setting
\[
 c_{n,j}=\frac{d_{n-j}(n)}{d_n(n)}\qquad(0\le j\le n).
\]
With this normalization, we can express $c_{n,j}/\binom nj$ by a beta integral,
as shown in Lemma~\ref{lem:moments}.
All finite sequences are extended by zero outside their indicated ranges.
For $0\le j,r\le n$, define
\[
 \widetilde M_n(x)=\sum_{j=0}^n(c_{n,j}^2-c_{n,j-1}c_{n,j+1})x^j,
 \qquad \Lambda_{n,j}=\frac{[x^j]\widetilde M_n(x)}{N_{n,j}},\qquad
 \eta_{n,r}=\sum_{j=0}^r(-1)^{r-j}\binom rj\Lambda_{n,j}.
\]
Then $\eta_{n,0}=\Lambda_{n,0}=1$ and
\[
 \widetilde M_n(x)=\frac{x^nM_n(1/x)}{d_n(n)^2}.
\]
The reversal preserves real-rootedness. We will compare signs after
mapping the zeros of $N_n(x)$ to $(0,1)$, then transfer the resulting
zero locations directly to $M_n(x)$.

Binomial inversion gives
\[
 \Lambda_{n,j}=\sum_{r=0}^j\binom jr\eta_{n,r}.
\]
Comparing coefficients and using $[x^j](x^rN_n^{(r)}(x)/r!)=\binom jr N_{n,j}$,
we obtain
\begin{equation}\label{eq:derivative-expansion}
 \widetilde M_n(x)=\sum_{r=0}^n\eta_{n,r}\frac{x^rN_n^{(r)}(x)}{r!}.
\end{equation}
The following estimate is proved in Section~\ref{sec:coefficient-comparison}.

\begin{proposition}\label{prop:moment-ratio}
For every integer $n\ge1$, we have $\eta_{n,r}>0$ for $0\le r\le n$.
Moreover, if $n\ge36$ and $4\le r\le n$,
\begin{equation}\label{eq:moment-ratio}
 \frac{\eta_{n,r}}{\eta_{n,r-1}}<\frac{r}{n-r+2}.
\end{equation}
\end{proposition}

We use the normalization
$P_m^{(\alpha,\beta)}(1)=\binom{m+\alpha}{m}$ for Jacobi polynomials
and set
\[
 J_n(w)=\frac{P_n^{(1,1)}(2w-1)}{n+1}.
\]
The identities used below for Jacobi polynomials with parameters greater
than $-1$ can be found in \cite{DLMF,Szego}.

Set
\[
 \begin{aligned}
 \widehat M_n(w)&=(1-w)^n\widetilde M_n\!\left(-\frac w{1-w}\right),\\
 \widehat N_n(w)&=(1-w)^nN_n\!\left(-\frac w{1-w}\right).
 \end{aligned}
\]
Since $\deg\widetilde M_n(x)=\deg N_n(x)=n$, these define polynomials of degree
at most $n$, with their values at $w=1$ given by continuity.

By the formula for $N_{n,k}$, $N_n(x)={}_2F_1(-n,-n-1;2;x)$.
We use Pfaff's transformation and the
hypergeometric formula for Jacobi polynomials to obtain
\[
 \widehat N_n(w)={}_2F_1(-n,n+3;2;w)=(-1)^nJ_n(w).
\]
By the orthogonality theorem for Jacobi polynomials, $J_n(w)$ has $n$ simple
zeros in $(0,1)$.

Write
$(a)_r=a(a+1)\cdots(a+r-1)$ and
$\fall{a}{r}=a(a-1)\cdots(a-r+1)$ for rising and falling factorials.

Define the polynomials $A_{n,k}(w)$ by $A_{n,-1}(w)=0$,
$A_{n,0}(w)=1$, and
\begin{equation}\label{eq:A-recurrence}
 A_{n,k+1}(w)=\left(2w-\frac{k+2}{n-k}(1-w)\right)A_{n,k}(w)
           -wA_{n,k-1}(w)\qquad(0\le k<n).
\end{equation}

\begin{lemma}\label{lem:jacobi-zero-identities}
For $n\ge1$,
\begin{equation}\label{eq:narayana-jacobi}
 \widehat N_n(w)=(-1)^nJ_n(w),\qquad A_{n,n}(w)=(n+1)J_n(w).
\end{equation}
At each zero $u$ of $J_n(w)$, put $x=-u/(1-u)$. Then
\begin{equation}\label{eq:derivative-values}
 A_{n,k}(u)=\frac{(1-u)^n x^{k+1}N_n^{(k+1)}(x)}
 {\fall{n}{k}\,u(1-u)\widehat N_n'(u)}
 \qquad(0\le k\le n),
\end{equation}
and
\begin{equation}\label{eq:zero-evaluation}
 \widehat M_n(u)=u(1-u)\widehat N_n'(u)
       \sum_{k=0}^{n-1}\frac{\eta_{n,k+1}}{k+1}\binom nk A_{n,k}(u).
\end{equation}
\end{lemma}
\begin{proof}
We write the hypergeometric differential equation for $N_n(x)$ and its
$r$th derivative as
\begin{align*}
 x(1-x)N_n''(x)+(2+2nx)N_n'(x)-n(n+1)N_n(x)&=0,\\
 x(1-x)N_n^{(r+2)}(x)+
 (r+2+(2n-2r)x)N_n^{(r+1)}(x)
 -(n-r)(n-r+1)N_n^{(r)}(x)&=0.
\end{align*}
At a zero $u$ of $J_n(w)$, set $x=-u/(1-u)$. The identity
\[
 (1-u)^nx N_n'(x)=u(1-u)\widehat N_n'(u)
\]
shows that the quotient in~\eqref{eq:derivative-values} equals $1$
when $k=0$. With $N_n(x)=0$ and $\widehat N_n'(u)\ne0$, the
differentiated equation shows that these quotients satisfy
\eqref{eq:A-recurrence}, starting with $A_{n,-1}(u)=0$ and
$A_{n,0}(u)=1$. Induction proves~\eqref{eq:derivative-values}.
The identity for $\widehat N_n(w)$ was established above.
The $r=0$ term in~\eqref{eq:derivative-expansion} vanishes.
Reindexing by $r=k+1$ and using
\[
 \frac{\fall{n}{k}}{(k+1)!}=\frac1{k+1}\binom nk
\]
with the formula for $A_{n,k}(u)$ gives~\eqref{eq:zero-evaluation}.
Since $N_n^{(n+1)}(x)=0$, we have $A_{n,n}(u)=0$ at every zero of $J_n(w)$.
By induction in~\eqref{eq:A-recurrence},
\[
 \deg A_{n,k}(w)\le k,\qquad A_{n,k}(1)=k+1\qquad(0\le k\le n).
\]
The polynomials $A_{n,n}(w)$ and $(n+1)J_n(w)$ agree, since both
have degree at most $n$, vanish at the $n$ zeros of $J_n(w)$, and
take the value $n+1$ at $w=1$.
\end{proof}

Define the positive weights
\[
 b_j=\frac{\eta_{n,j+1}}{j+1}\binom nj\qquad(0\le j<n).
\]
For $n\ge36$, \eqref{eq:zero-evaluation} reduces the sign comparison
at a zero $u$ of $J_n(w)$ to $\sum_{j=0}^{n-1}b_jA_{n,j}(u)>0$.
To apply summation by parts, define the partial sums
\[
 V_{n,k}(w)=\sum_{j=0}^{k-1}A_{n,j}(w)\qquad(0\le k\le n),
\]
with $V_{n,0}(w)=0$.
The required lower bound for $V_{n,k}(u)$ is proved in
Section~\ref{sec:partial-sum-bounds}.

\begin{theorem}\label{thm:all-zeros}
For every integer $n\ge36$, every zero $u$ of $J_n(w)$, and
every $1\le k\le n$, one has $V_{n,k}(u)>(n+1)/(2(n+3))$.
\end{theorem}

Since
\[
 \frac{b_j}{b_{j-1}}
 =\frac{n-j+1}{j+1}\frac{\eta_{n,j+1}}{\eta_{n,j}},
\]
Proposition~\ref{prop:moment-ratio} implies $b_j<b_{j-1}$
for $n\ge36$ and $3\le j<n$.
\begin{lemma}\label{lem:weight-shape}
For $n\ge36$, the weights satisfy
\[
 0<b_0<b_1<b_2>b_3>\cdots>b_{n-1}>0,
 \qquad
 (b_1-b_0)+3(b_2-b_1)<\frac{b_2}{3}.
\]
\end{lemma}
\begin{proof}
Put $T=8n^3-19n^2+48n-27>0$. Lemma~\ref{lem:moments} and the definitions of $\Lambda_{n,j}$ and $\eta_{n,r}$ give
\[
 \frac{b_1-b_0}{b_2}
 =\frac{6(2n-3)(n^2+3n-1)}{(4n-1)T}>0,\qquad
 \frac{b_2-b_1}{b_2}
 =\frac{3(5n-1)(7n-3)}{(4n-1)T}>0.
\]
Finally,
\[
 \frac{(b_1-b_0)+3(b_2-b_1)}{b_2}
 =\frac{3(n^2+28n-15)}{T}<\frac13,
\]
because
\[
 T-9(n^2+28n-15)=4(2n-1)(n^2-3n-27)>0.\qedhere
\]
\end{proof}

We now deduce the required sign comparison.

\begin{proposition}\label{prop:reduction}
Let $n\ge36$. At every zero $u$ of $J_n(w)$,
\begin{equation}\label{eq:zero-sign}
 \frac{\widehat M_n(u)}{u(1-u)\widehat N_n'(u)}
 >\frac{n-3}{6(n+3)}b_2>0.
\end{equation}
Writing $0<u_1<\cdots<u_n<1$ for the zeros of $J_n(w)$, the polynomial
$\widehat M_n(w)$ has exactly one simple zero in each of
$(0,u_1),(u_1,u_2),\ldots,(u_{n-1},u_n)$.
\end{proposition}
\begin{proof}
Fix a zero $u$ of $J_n(w)$, and put $c=(n+1)/(2(n+3))$.
Theorem~\ref{thm:all-zeros} gives $V_{n,k}(u)>c$ for $1\le k\le n$.
Summation by parts gives
\[
 \sum_{j=0}^{n-1}b_jA_{n,j}(u)
 =b_{n-1}V_{n,n}(u)
  +\sum_{j=1}^{n-1}(b_{j-1}-b_j)V_{n,j}(u).
\]
By Lemma~\ref{lem:weight-shape}, only the terms with $j=1,2$
have negative coefficients; the positive coefficients sum to $b_2$.
Also, $V_{n,1}(u)=1$ and
$V_{n,2}(u)=1-2/n+2(n+1)u/n<3$. Hence
\[
 \sum_{j=0}^{n-1}b_jA_{n,j}(u)
 \ge cb_2-(b_1-b_0)-3(b_2-b_1)
 >b_2\left(c-\frac13\right)>0,
\]
where $c-1/3=(n-3)/(6(n+3))>0$.
Substituting into~\eqref{eq:zero-evaluation} proves~\eqref{eq:zero-sign}.
Write the zeros of $J_n(w)$ as $0<u_1<\cdots<u_n<1$.
Since all its zeros lie below $1$ and $J_n(1)=1$, the leading
coefficient of $J_n(w)$ is positive. Hence
$\operatorname{sgn}J_n'(u_i)=(-1)^{n-i}$.
Together with $\widehat N_n'(u_i)=(-1)^nJ_n'(u_i)$, this gives
$\operatorname{sgn}\widehat M_n(u_i)=(-1)^i$.
As $\widehat M_n(0)=1$, there is a zero in each of
$(0,u_1),(u_1,u_2),\ldots,(u_{n-1},u_n)$.
Since the polynomial has degree at most $n$, these are all its zeros
and each is simple.
\end{proof}

\section{Coefficient estimates}\label{sec:coefficient-comparison}

Fix an integer $n\ge1$. Define the weight
\[
 f_n(x)=\frac{2^{2n+1}}{\mathrm B(1/2,2n+1)}
 (x-1)^{-1/2}(x+1)^{-2n-3/2}\qquad(x>1),
\]
and its moments
\[
 m_{n,j}=\int_1^\infty x^j f_n(x)\,dx\qquad(0\le j\le n+1).
\]

\begin{lemma}\label{lem:moments}
For $0\le j\le n+1$,
\[
 m_{n,j}=\sum_{r=0}^j\binom jr
             \frac{2^r(1/2)_r}{\fall{(2n)}{r}}.
\]
Moreover, $c_{n,j}=\binom nj m_{n,j}$ for $0\le j\le n$.
\end{lemma}
\begin{proof}
Substituting $x=1+2t$ and expanding $(1+2t)^j$, we obtain
\begin{align*}
 m_{n,j}
 &=\frac1{\mathrm B(1/2,2n+1)}
   \int_0^\infty(1+2t)^j t^{-1/2}(1+t)^{-2n-3/2}\,dt\\*
 &=\sum_{r=0}^j\binom jr2^r
   \frac{\mathrm B(r+1/2,2n+1-r)}{\mathrm B(1/2,2n+1)}\\*
 &=\sum_{r=0}^j\binom jr
   \frac{2^r(1/2)_r}{\fall{(2n)}{r}}.
\end{align*}
These integrals converge because $j\le n+1<2n+1$.
In particular, $\int_1^\infty f_n(x)\,dx=m_{n,0}=1$.

In the defining sum for $d_{n-j}(n)$, put $k=n-r$ and use
$d_n(n)=2^{-n}\binom{2n}n$. The identities
\[
 \frac{\binom{2n-r}{n-r}\binom{n-r}{n-j}}{\binom{2n}n}
 =\binom nj\frac{\fall{j}{r}}{\fall{(2n)}{r}},\qquad
 2^{-r}\binom{2r}r=\frac{2^r(1/2)_r}{r!}
\]
give $c_{n,j}=\binom nj m_{n,j}$.
\end{proof}

Write $d\mu_n(x,y)=f_n(x)f_n(y)\,dx\,dy$ on $(1,\infty)^2$.

\begin{lemma}\label{lem:eta-integrals}
For $1\le r\le n$,
\begin{equation}\label{eq:positive-eta}
 \eta_{n,r}=\frac1{4(n+1)}\int_{(1,\infty)^2} (xy-1)^{r-1}
 \bigl(2(2n-r+2)(xy-1)+r^2(x-y)^2\bigr)\,d\mu_n(x,y)>0.
\end{equation}
\end{lemma}
\begin{proof}
Put $t=xy-1$, $e=(x-y)^2$, and $g(t)=t^{r-1}$.
For $0\le j\le n$, Lemma~\ref{lem:moments} and symmetry in $x,y$ give
\[
 \Lambda_{n,j}=\int_{(1,\infty)^2}\left((1+t)^j
 -\frac{j(n-j)}{2(n+1)}e(1+t)^{j-1}\right)\,d\mu_n.
\]
Taking the $r$th finite difference and using the first two derivatives
of the binomial expansion of $t^r$, we obtain
\[
 \eta_{n,r}=\int_{(1,\infty)^2}\left(tg(t)-\frac{re}{2(n+1)}
 \bigl((n-r)g(t)-g'(t)\bigr)\right)\,d\mu_n.
\]
The weight satisfies
\[
 ((x^2-1)f_n(x))'=((2n+1)-2nx)f_n(x).
\]
Apply integration by parts to $(x-y)g(t)$ in $x$ and $y$, and subtract
the resulting identities. Using $tg'(t)=(r-1)g(t)$, we obtain
\[
 \int_{(1,\infty)^2}\bigl(2tg(t)-(2n-r)eg(t)+2eg'(t)\bigr)\,d\mu_n=0.
\]
For each fixed $y>1$, the boundary terms are $O((x-1)^{1/2})$
at $1$ and $O(x^{-n})$ at infinity, and hence vanish; the same holds
with $x$ and $y$ interchanged. The polynomial factors in these integrations
have degree at most $r+1\le n+1$ in each variable, so Lemma~\ref{lem:moments} ensures
absolute integrability and justifies integration in either order.
Eliminating the integral containing $g'(t)$ proves~\eqref{eq:positive-eta}.
Its integrand is positive for $x,y>1$. The argument includes $r=1$,
when $g(t)=1$ and $g'(t)=0$.
\end{proof}

\begin{proof}[Proof of Proposition~\ref{prop:moment-ratio}]
We have $\eta_{n,0}=1$, and Lemma~\ref{lem:eta-integrals} gives
$\eta_{n,r}>0$ for $1\le r\le n$. For the ratio bound, let
$n\ge36$ and $4\le r\le n$, and put
\[
 \ell=2n-r,\qquad h=n-r,\qquad t=xy-1,\qquad e=(x-y)^2.
\]
Write $I_j=\int_{(1,\infty)^2} t^j\,d\mu_n$ and $E_j=\int_{(1,\infty)^2} et^j\,d\mu_n$.
The integration-by-parts identity in Lemma~\ref{lem:eta-integrals} gives
\[
 (2n-j)E_{j-1}=2I_j+2(j-1)E_{j-2}\qquad(2\le j\le r).
\]
Applying the same integration-by-parts formula in $x$ to $t^{r-1}$
and $yt^{r-1}$, respectively, and using symmetry in $x,y$, we obtain
\begin{align*}
 \int_{(1,\infty)^2} yt^{r-1}\,d\mu_n&=\frac{2n+1}{\ell+1}I_{r-1},\\*
 (\ell+1)(I_r+I_{r-1})
 &=(2n+1)\int_{(1,\infty)^2} yt^{r-1}\,d\mu_n-\frac{r-1}{2}E_{r-2}.
\end{align*}
The boundary and integrability estimates are the same as in that lemma.
Eliminating the mixed moment yields
\[
 (\ell+1)^2I_r=r(2\ell+r+2)I_{r-1}
              -\frac{(r-1)(\ell+1)}2E_{r-2}.
\]
The integration-by-parts identity at indices $r-1$ and $r$ also gives
\[
 E_{r-2}=\frac{2I_{r-1}+2(r-2)E_{r-3}}{\ell+1},\qquad
 E_{r-1}=\frac{2I_r+2(r-1)E_{r-2}}{\ell}.
\]
We first substitute the formula for $E_{r-2}$ into that for $I_r$,
and then use the formula for $E_{r-1}$. Substituting the resulting
expressions into~\eqref{eq:positive-eta} at indices $r$ and $r-1$ gives
\[
 r\eta_{n,r-1}-(n-r+2)\eta_{n,r}
 =\frac{2P_r(h)I_{r-1}+(r-1)(r-2)Q_r(h)E_{r-3}}
 {2(n+1)\ell(\ell+1)^2},
\]
where
\begin{align*}
 P_r(h)&=2(r^2-3r-1)h^3+2(r^3-6r^2-4r-3)h^2\\*
       &\quad +(2r^4-16r^3-7r^2-5r-4)h\\*
       &\quad +r(r+1)(r^3-7r^2+2r-2),\\
 Q_r(h)&=4h^3+12h^2+2(r-2)(r^2-3r-2)h\\*
       &\quad +r(r-4)(r-1)(r+1).
\end{align*}
Clearly $Q_r(h)\ge0$ for $r\ge4$ and $h\ge0$.
For $r\ge9$, put $s=r-9\ge0$. Then
\begin{align*}
 P_r(h)&=(2s^2+30s+106)h^3\\*
       &\quad +2(s^3+21s^2+131s+204)h^2\\*
       &\quad +(2s^4+56s^3+533s^2+1813s+842)h\\*
       &\quad +(s+9)(s+10)(s^3+20s^2+119s+178)>0.
\end{align*}
For $4\le r\le8$, we have $h=n-r\ge28$. Substitution into the
formula defining $P_r(h)$ gives
\[
 P_r(h)\ge6h^3-102h^2-1150h-1260
 \ge h(66h-1150)-1260>0.
\]
Since $I_{r-1}>0$ and $E_{r-3}\ge0$, the desired ratio bound follows.
\end{proof}

\section{Bounds for the partial sums}\label{sec:partial-sum-bounds}

\subsection{Identities for the partial sums}\label{sec:identities-at-zeros}

Define
\begin{equation}\label{eq:normalized-remainder}
 R_{n,k}(w)=1-\frac{(n+3)(1-w)}{n+1}V_{n,k}(w)\qquad(0\le k\le n).
\end{equation}
For $w\ne1$,
\[
 V_{n,k}(w)=\frac{n+1}{(n+3)(1-w)}\bigl(1-R_{n,k}(w)\bigr).
\]
For $0<u<1$, this gives
\begin{equation}\label{eq:prefix-remainder}
 (1-u)\left(V_{n,k}(u)-\frac{n+1}{2(n+3)}\right)
 =\frac{n+1}{n+3}\left(\frac{1+u}{2}-R_{n,k}(u)\right).
\end{equation}
For $n\ge36$, Theorem~\ref{thm:all-zeros} is therefore equivalent to
\[
 R_{n,k}(u)<\frac{1+u}{2}
\]
at every zero $u$ of $J_n(w)$ and for $1\le k\le n$.

\begin{lemma}\label{lem:boundary-recurrence}
For $n\ge1$ and $0\le k\le n$,
\begin{equation}\label{eq:R-boundary}
 R_{n,k}(w)=\frac{(n-k+1)A_{n,k}(w)-w(n-k)A_{n,k-1}(w)}{n+1}.
\end{equation}
\end{lemma}
\begin{proof}
The right-hand side of~\eqref{eq:R-boundary} equals $1$ when $k=0$.
By~\eqref{eq:A-recurrence}, the difference between the right-hand sides
of~\eqref{eq:R-boundary} at indices $k+1$ and $k$ is
\[
 -\frac{n+3}{n+1}(1-w)A_{n,k}(w)\qquad(0\le k<n).
\]
Summing these differences and using~\eqref{eq:normalized-remainder}
proves the identity.
\end{proof}

For $n\ge1$, we take $k=0,1,n$ in~\eqref{eq:R-boundary} and use
\eqref{eq:A-recurrence} and~\eqref{eq:narayana-jacobi} to obtain
\begin{equation}\label{eq:R-initial}
 R_{n,0}(w)=1,\qquad R_{n,1}(w)=\frac{(n+3)w-2}{n+1},\qquad R_{n,n}(w)=J_n(w).
\end{equation}

For $w\ne1$, subtract adjacent instances of
\eqref{eq:normalized-remainder}, solve for $A_{n,k}(w)$ and
$A_{n,k-1}(w)$, and substitute into~\eqref{eq:R-boundary}. This gives
\begin{equation}\label{eq:R-recurrence}
 R_{n,k+1}(w)=a_k(w)R_{n,k}(w)-\beta_k(w)R_{n,k-1}(w)
 \qquad(1\le k<n),
\end{equation}
where
\[
 a_k(w)=\frac{(2n-k+3)w-k-2}{n-k+1},\qquad
 \beta_k(w)=\frac{(n-k)w}{n-k+1}.
\]
Both sides of~\eqref{eq:R-recurrence} are polynomials in $w$, so the
recurrence also holds at $w=1$.

At a zero $u$ of $J_n(w)$, we have $R_{n,n}(u)=0$
by~\eqref{eq:R-initial}. Substituting into~\eqref{eq:normalized-remainder},
we obtain
\[
 V_{n,n}(u)=\frac{n+1}{(n+3)(1-u)}.
\]
For $0\le k\le n$, we have
\[
 R_{n,k}(u)=\frac{V_{n,n}(u)-V_{n,k}(u)}{V_{n,n}(u)}
       =\frac{\sum_{j=k}^{n-1}A_{n,j}(u)}
              {\sum_{j=0}^{n-1}A_{n,j}(u)}.
\]

\subsection{Bounds from the recurrence}

We distinguish the cases $k\le(n+3)u-2$ and $k>(n+3)u-2$.
The first is equivalent to $(n+3)(1-u)\le n-k+1$, the condition
needed for the quadratic estimate below. In the second case,
$u<(k+2)/(n+3)$; we use this inequality to estimate $u^k$ in
Proposition~\ref{prop:gaussian-bound}.

\begin{proposition}\label{prop:recurrence-bound}
Let $n\ge1$, $0<u<1$, and $1\le k\le n$. If $k\le(n+3)u-2$, then
\[
 |R_{n,k}(u)|<\sqrt u<\frac{1+u}{2}.
\]
\end{proposition}
\begin{proof}
Put $N=n+3$, $q=1-u$, and $\lambda=Nq$. For $0\le j\le k$, write
$R_j=R_{n,j}(u)$, and for $1\le j\le k$, put
$m_j=n-j+1$ and $\Delta_j=R_j-R_{j-1}$.
The hypothesis gives $\lambda\le m_k\le m_j\le n$ for $1\le j\le k$,
and the recurrence yields
\[
 \Delta_{j+1}=\frac{(m_j-1)u}{m_j}\Delta_j
              -\frac{\lambda}{m_j}R_j\qquad(1\le j<k).
\]
Also, $R_0=1$ and $R_1=1-\lambda/(n+1)$.

First suppose $0<\lambda\le1$, and set
\[
 E_j=R_j^2+\frac{m_j+1-\lambda}{\lambda}\Delta_j^2\ge R_j^2.
\]
Substituting the recurrence gives
\[
 E_{j+1}=\left(1-\frac{\lambda}{m_j}\right)R_j^2
 +\frac{(m_j-1)^2u^2}{\lambda m_j}\Delta_j^2\le uE_j,
\]
since $m_j<N$ and
$(m_j-1)^2u/m_j\le m_j-1<m_j+1-\lambda$.
As $E_1=1-\lambda/(n+1)<u$, we obtain $R_k^2\le E_k\le E_1<u$.

Now suppose $\lambda>1$. Put $d_j=1-(N-m_j)q$ and
\[
 \mathcal E_j=\frac{n+1}{m_j}
 \left(R_j^2+\frac{d_j}{\lambda}R_j\Delta_j
                  +\frac{m_ju}{\lambda}\Delta_j^2\right).
\]
For $\lambda\le m\le n$ and $a=N-m\ge3$, we have
\[
 N\bigl(maqu-(1-aq)^2\bigr)
 =m(a-1)+a(\lambda-1)(m+1-\lambda)>0.
\]
Taking $m=m_j$ and using $\lambda=Nq$, we obtain
$d_j^2<m_j(N-m_j)u\lambda/N$, and hence
\[
 1-\frac{d_j^2}{4m_ju\lambda}>\frac{3N+m_j}{4N}.
\]
Completing the square now gives
\begin{align*}
 \mathcal E_j
 &=\frac{n+1}{m_j}\left[
   \frac{m_ju}{\lambda}
   \left(\Delta_j+\frac{d_j}{2m_ju}R_j\right)^2
   +\left(1-\frac{d_j^2}{4m_ju\lambda}\right)R_j^2
   \right]\\*
 &\ge\frac{(n+1)(3N+m_j)}{4m_jN}R_j^2
 \ge\left(1+\frac{n+9}{4nN}\right)R_j^2,
\end{align*}
where the last inequality uses $m_j\le n$.
For $1\le j<k$, substitution in the recurrence gives
\[
 \mathcal E_{j+1}=u\mathcal E_j
 -\frac{(n+1)u}{\lambda m_j}\Delta_j(qR_j+u\Delta_j).
\]
Since $\Delta_j(qR_j+u\Delta_j)\ge-q^2R_j^2/(4u)$ and
$\lambda=Nq$, it follows that
\[
 \mathcal E_{j+1}
 \le\left(u+\frac{(n+1)q}{4m_jN}\right)\mathcal E_j
 \le\mathcal E_j.
\]
Here the factor is less than $1$ because $m_j\ge1$ and $N>n+1$.
Since $\mathcal E_1=u$, the preceding bounds give
\[
 R_k^2
 \le\frac{\mathcal E_k}{1+\frac{n+9}{4nN}}
 \le\frac{\mathcal E_1}{1+\frac{n+9}{4nN}}
 =\frac{u}{1+\frac{n+9}{4nN}}<u.
\]
Finally, $0<u<1$ implies $2\sqrt u<1+u$.
\end{proof}

\subsection{Jacobi polynomials and Gaussian quadrature}

\begin{proposition}\label{prop:gaussian-bound}
Let $n\ge2$ and $1\le k<n$. At every zero $u$ of $J_n(w)$,
\begin{equation}\label{eq:gaussian-bound}
 |R_{n,k}(u)|^2
 \le\frac{(n-k)(n-k+1)(n+3)}{(n+1)(2n-k+3)}\,u^k.
\end{equation}
\end{proposition}
\begin{proof}
For $0\le k<n$, put
\[
 Q_{n,k}(w)=P_{n-k-1}^{(2,k+2)}(2w-1),\qquad
 Y_{n,k}(w)=w^kQ_{n,k}(w),
\]
and set $Y_{n,n}(w)=0$. We first show that
\begin{equation}\label{eq:jacobi-remainder}
 R_{n,k}(u)=\frac{n+3}{n+1}
              \frac{u^kQ_{n,k}(u)}{J_n'(u)}\qquad(0\le k<n).
\end{equation}
Eliminating $P_{m-1}^{(2,\beta)}(x)$ from the
contiguous relations \cite[(18.9.5), (18.9.6)]{DLMF}, and then applying
(18.9.6) once more, gives
\[
 (m+1)P_{m+1}^{(2,\beta-1)}(x)
 =\left((2m+\beta+3)\frac{1+x}{2}-\beta\right)P_m^{(2,\beta)}(x)
 -(m+2)\frac{1+x}{2}P_{m-1}^{(2,\beta+1)}(x).
\]
Here $m\ge0$, $\beta>0$, and $P_{-1}(x)=0$; the case $m=0$
also follows directly from the formulas for $P_0(x)$ and $P_1(x)$.
Taking $m=n-k-1$, $\beta=k+2$, $x=2w-1$, and multiplying by $w^k$ gives
\[
 (n-k+1)Y_{n,k+1}(w)
 =\bigl((2n-k+3)w-k-2\bigr)Y_{n,k}(w)
       -(n-k)wY_{n,k-1}(w)
 \qquad(1\le k<n).
\]
At a zero $u$ of $J_n(w)$, the sequences $Y_{n,k}(u)$ and $R_{n,k}(u)$
therefore satisfy the same recurrence, with
$Y_{n,n}(u)=R_{n,n}(u)=0$. Since $\beta_j(u)>0$ for $1\le j<n$,
we can determine each sequence backward from its value at $n-1$.
As $Y_{n,n-1}(u)=u^{n-1}\ne0$, the two sequences are proportional.
By the differentiation formula for Jacobi polynomials,
\[
 Y_{n,0}(w)=P_{n-1}^{(2,2)}(2w-1)=\frac{n+1}{n+3}J_n'(w).
\]
Together with $R_{n,0}(u)=1$ and $J_n'(u)\ne0$, this determines
the proportionality constant and proves~\eqref{eq:jacobi-remainder}.

Apply Gaussian quadrature for $w(1-w)\,dw$ on $[0,1]$ at the
zeros $u_1,\ldots,u_n$ of $J_n(w)$. The formula is exact through degree
$2n-1$ and has positive weights
\begin{equation}\label{eq:gaussian-weight}
 \omega_i=\frac{1}{(n+1)(n+2)u_i(1-u_i)J_n'(u_i)^2}.
\end{equation}
Indeed, using the norm formula for Jacobi polynomials, we obtain
\[
 \int_0^1w(1-w)J_j(w)^2\,dw
 =\frac{1}{(2j+3)(j+1)(j+2)}.
\]
The ratio of leading coefficients of $J_n(w)$ and $J_{n-1}(w)$ is
$2(2n+1)/(n+2)$. At a zero $u_i$, the recurrence and differentiation
formulas for Jacobi polynomials give
$J_{n-1}(u_i)=2u_i(1-u_i)J_n'(u_i)/n$.
Substituting these identities into the Christoffel--Darboux formula,
we obtain~\eqref{eq:gaussian-weight}.

Fix $1\le k<n$. The polynomial $w^{k+1}(1-w)Q_{n,k}(w)^2$
has degree $2n-k\le2n-1$, so exactness of Gaussian quadrature gives
\begin{equation}\label{eq:gaussian-quadrature}
 \sum_{i=1}^n\omega_i u_i^{k+1}(1-u_i)Q_{n,k}(u_i)^2
 =\int_0^1 w^{k+2}(1-w)^2Q_{n,k}(w)^2\,dw.
\end{equation}
By~\eqref{eq:jacobi-remainder} and~\eqref{eq:gaussian-weight},
\[
 \omega_i u_i^{k+1}(1-u_i)Q_{n,k}(u_i)^2
 =\frac{n+1}{(n+2)(n+3)^2}\frac{R_{n,k}(u_i)^2}{u_i^k}.
\]
The Jacobi norm formula \cite[Table 18.3.1]{DLMF} yields
\[
 \int_0^1 w^{k+2}(1-w)^2Q_{n,k}(w)^2\,dw
 =\frac{(n-k)(n-k+1)}{(2n-k+3)(n+2)(n+3)}.
\]
Substituting these formulas into~\eqref{eq:gaussian-quadrature}, we obtain
\begin{equation}\label{eq:remainder-square-sum}
 \sum_{i=1}^n\frac{R_{n,k}(u_i)^2}{u_i^k}
 =\frac{(n-k)(n-k+1)(n+3)}{(n+1)(2n-k+3)}
 \qquad(1\le k<n).
\end{equation}
Every summand is nonnegative. Retaining the summand at any prescribed
zero $u=u_i$ and multiplying by $u^k>0$ proves~\eqref{eq:gaussian-bound}.
\end{proof}

\subsection{Partial sums at the zeros of \texorpdfstring{$J_n(w)$}{Jn(w)}}

\begin{proof}[Proof of Theorem~\ref{thm:all-zeros}]
Let $n\ge36$ and let $u$ be a zero of $J_n(w)$.
The case $k=1$ follows from $V_{n,1}(u)=1$.
For $k=n$, use $R_{n,n}(u)=0$ in~\eqref{eq:prefix-remainder}.
For $k=2$, the recurrence for $A_{n,k}(u)$ gives
\[
 V_{n,2}(u)=1-\frac2n+\frac{2(n+1)}n u>1-\frac2n,
 \qquad
 1-\frac2n-\frac{n+1}{2(n+3)}
 =\frac{(n+4)(n-3)}{2n(n+3)}>0.
\]
We may therefore assume $3\le k<n$. By~\eqref{eq:prefix-remainder},
it suffices to prove $R_{n,k}(u)<(1+u)/2$.

If $k\le(n+3)u-2$, Proposition~\ref{prop:recurrence-bound} gives
$|R_{n,k}(u)|<\sqrt u<(1+u)/2$.

If $k>(n+3)u-2$, put $a=k+2\ge5$ and $b=n-k+1\ge2$.
Then $u<a/(a+b)$. Proposition~\ref{prop:gaussian-bound} and
$(b-1)(a+b)<b(a+2b-1)$ give
\begin{align*}
 |R_{n,k}(u)|^2
 &<\frac{b^2}{n+1}\left(1+\frac ba\right)^{-(a-2)}\\*
 &<\frac{2a^2}{(n+1)(a-2)(a-3)}
 \le\frac{25}{3(n+1)}<\frac14.
\end{align*}
Here we retained the quadratic term in the binomial expansion and used
$a-2\ge3a/5$ and $a-3\ge2a/5$.
Thus $|R_{n,k}(u)|<1/2<(1+u)/2$.
\end{proof}

\section{Proofs of the main results}\label{sec:main-proofs}

\begin{proof}[Proof of Theorem~\ref{thm:main}]
Let $n\ge36$. By Proposition~\ref{prop:reduction}, the simple zeros
$v_1,\ldots,v_n$ of $\widehat M_n(w)$ and $u_1,\ldots,u_n$ of $J_n(w)$ satisfy
\[
 0<v_1<u_1<v_2<u_2<\cdots<v_n<u_n<1.
\]
The coefficient formula for $N_n(x)$ gives $x^nN_n(1/x)=N_n(x)$.
For $0<w<1$, set $g(w)=-(1-w)/w$. The definitions then yield
\[
 \widehat M_n(w)=\frac{(-1)^nw^n}{d_n(n)^2}M_n(g(w)),\qquad
 \widehat N_n(w)=(-1)^nw^nN_n(g(w)).
\]
Since $g'(w)=1/w^2>0$ and $g$ maps $(0,1)$ onto $(-\infty,0)$,
these identities transfer the zeros to $M_n$ and $N_n$ without
changing their order or multiplicities. Thus $\mu_{n,i}=g(v_i)$ and
$\nu_{n,i}=g(u_i)$ have the required strict interlacing.

It remains to consider $1\le n\le35$. Put
\[
 F_n(x)=16^n x^nM_n(1/x)=16^nd_n(n)^2\widetilde M_n(x).
\]
This polynomial has integer coefficients by~\eqref{eq:bm-coefficients}. Put
\[
 H_n(x)=\operatorname{lc}(F_n(x))N_n(x)-F_n(x).
\]
Since $N_n(x)$ is monic, $\deg H_n(x)<n$. Starting with $N_n(x),H_n(x)$,
form successive negatives of Euclidean remainders, clearing denominators
by positive factors and dividing by positive contents. Exact rational arithmetic
gives degrees $n,n-1,\ldots,0$ and positive leading coefficients
throughout this sequence for every $1\le n\le35$.
Let $P,Q,R$ be three consecutive polynomials in this sequence,
of degrees $d+1,d,d-1$, respectively. Positive normalization gives
$P=qQ-\gamma R$ with $\gamma>0$.
Suppose that $Q$ has simple real zeros $\xi_1<\cdots<\xi_d$
and that the zeros of $R$ strictly interlace them. Since the leading
coefficient of $R$ is positive,
\[
 \operatorname{sgn}P(\xi_i)
 =-\operatorname{sgn}R(\xi_i)=(-1)^{d-i+1}
 \qquad(1\le i\le d).
\]
Hence $P$ has a zero in each $(\xi_i,\xi_{i+1})$.
For sufficiently negative $x$, the sign of $P(x)$ is $(-1)^{d+1}$,
opposite to its sign $(-1)^d$ at $\xi_1$.
For sufficiently positive $x$, we have $P(x)>0$, whereas $P(\xi_d)<0$.
Thus there is also a zero in each of $(-\infty,\xi_1)$ and
$(\xi_d,\infty)$. These $d+1$ distinct zeros exhaust the degree of $P$,
so they are simple and strictly interlace the zeros of $Q$.
Starting with the final positive constant and the preceding linear
polynomial, induction applies throughout the sequence. In particular,
\[
 H_n(\nu_{n,i})N_n'(\nu_{n,i})>0\qquad(1\le i\le n).
\]
Since $F_n(\nu_{n,i})=-H_n(\nu_{n,i})$ and
$F_n(0)=16^nd_n(n)^2>0$,
there is a zero of $F_n(x)$ between each pair of consecutive zeros
of $N_n(x)$, and one between $\nu_{n,n}$ and $0$.
Since $\deg F_n(x)=n$, these $n$ zeros account for its full degree,
so each is simple and there are no others.

The reciprocity of $N_n(x)$ gives $\nu_{n,i}=1/\nu_{n,n+1-i}$.
By the definition of $F_n(x)$, taking reciprocals gives the zeros
of $M_n(x)$ without changing their multiplicities. Applying this
to the intervals established for $F_n(x)$, we see that, for
$1\le n\le35$, $M_n(x)$ has exactly one simple zero in each of
\[
 (-\infty,\nu_{n,1}),\qquad
 (\nu_{n,i},\nu_{n,i+1})\quad(1\le i<n).
\]
This proves the theorem.
\end{proof}

\begin{proof}[Proof of Corollary~\ref{thm:infinite-log-concavity}]
For $n\ge1$, Theorem~\ref{thm:main} and the positive leading
coefficient $d_n(n)^2$ show that all coefficients of $M_n(x)$ are positive.
Br\"and\'en's theorem \cite[Section~3]{Branden} states that if a finite
nonnegative sequence has a generating polynomial with only nonpositive
real zeros, then so does its image under $\mathcal L$.
At every iteration the leading and constant coefficients are positive
squares. Thus the degree remains $n$, and the transformed polynomial
has only negative zeros and strictly positive coefficients.
Starting with $M_n(x)$, induction therefore
proves the asserted real-rootedness and strict coefficient positivity
for every iterate $r\ge1$. Together with the positivity of the original
coefficients in~\eqref{eq:bm-coefficients}, this proves infinite
log-concavity. The case $n=0$ is immediate.
\end{proof}

\section*{Acknowledgements}

Matthew H.~Y.~Xie was
supported by the National Natural Science Foundation of China (Grant
No.~12271403). Philip B.~Zhang was supported by the National Natural Science
Foundation of China (Grant No.~12171362) and Natural Science Foundation of
Tianjin Municipality (Grant No.~25JCYBJC00430).

\section*{Declaration of Generative AI}
During the preparation of this manuscript, the authors used generative AI tools
to assist in exploring possible approaches, checking technical details,
and improving the exposition. All mathematical arguments, computations,
proofs, and results were independently verified by the authors. The authors
take full responsibility for the content of this manuscript.

\end{document}